\documentclass[leqno]{article}
\usepackage{latexsym}
\usepackage{amsmath,amsthm,amstext,amsbsy,amssymb,amsfonts,txfonts, amscd,}
\usepackage{etoolbox}
\usepackage{amssymb}
\usepackage{amsmath}
\usepackage{textcomp}
\usepackage[utf8]{inputenc}
\usepackage[T1]{fontenc}
\usepackage{fancyhdr}
 \usepackage{epsfig}
 \usepackage{graphics}
 \usepackage{graphicx}
\usepackage{amstext}
\usepackage{amsthm}
\usepackage{amsxtra}
\usepackage{multirow}
\newtheorem{The}{Theorem}
\newtheorem{Cor}[The]{Corollary}

\newtheorem{Pro}[The]{Proposition}
\newtheorem{Lem}[The]{Lemma}
\numberwithin{equation}{section}
\numberwithin{The}{section}
\newcommand{\address}[1]{\vskip3mm\noindent#1}

\newcommand{\N}{{\mathbb{N}}}
\newcommand{\C}{{\mathbb{C}}}

\newcommand{\Z}{{\mathbb{Z}}}

\newcounter{subeqn}\renewcommand{\thesubeqn}{\theequation\alph{\subeqn}}
\newcommand{\subeqn}{
\refstepcounter{subeqn}
\tag{\thesubeqn}}
\title{On the Power of Integers and Conductors of Quadratic Fields\\}
\author{Nihal Bircan $^{\ast}$ and Michael E. Pohst$^{\ast\ast}$}

\begin{document}

\maketitle
\begin{center}
 \address {$^{\ast}$ Berlin University of Technology, Institute for Mathematics, MA $8-1$, Strasse des $17.$ Juni $136$ D-$10623$ Berlin, Germany\\
\c{C}ank{\i}r{\i} Karatekin University, Department of Mathematics, TR $18100$, \c{C}ank{\i}r{\i}, Turkey\\

bircan@math.tu-berlin.de\\
$^{\ast\ast}$ Berlin University of Technology, Institute for Mathematics, MA $8-1$, D-$10623$ Berlin, Germany\\
pohst@math.tu-berlin.de }
\end{center}

\begin{abstract}
We consider the integers $\alpha$ of the quadratic field $ \mathbb{Q} (\sqrt{d}$ $ )$ where $d\in \Z$ is square-free and $d\equiv 1,2,3 \pmod 4$. Let $p$ be an odd prime. 
Using the embedding into $ \text{GL}(2,\mathbb{Z})$ we obtain bounds for the first $\nu\in\N$ such that $\alpha^\nu\equiv 1\pmod p.$ For the 
conductor $f$, we then study the first integer $n=n(f)$ such that $\alpha^n\in\mathcal{O}_{f}$. We obtain bounds for $n(f)$ and for $n(fp^{k})$.
The most interesting case is that $\alpha$ is the fundamental unit of $ \mathbb{Q} (\sqrt{d}$ $ )$.
\end{abstract}
{\bf $2010$ Mathematics Subject Classification :}{$11$R$11$, $11$R$04$, $42$C$05$}

\section{Introduction}
We study the quadratic field $\mathbb{Q}(\sqrt{d}$ $)$ where $d\in \Z$ is square-free. We write $d=4q+r$ where $r=1,2,3$. The 
algebraic integers $\alpha$ of  $ \mathbb{Q} (\sqrt{d}$ $ )$ are given by
\begin{align} \label{1.1} \alpha =\begin{cases}
a+b\sqrt{d} , \text{\ \ } a,b\in \mathbb{Z} & \text{if } r=2,3\\ 
\tfrac{1}{2} (a+b\sqrt{d} ),\ \ a,b\in \mathbb{Z} ,\ a+b\in 2\mathbb{Z} &
\text{if } r=1.
\end{cases}\end{align}
Throughout the paper we consider any quadratic field and its integers $\alpha$ of any norm$\neq 0$. Let $p$ be an odd prime. First we study the problem to find small 
exponents $n$ such that $\alpha^{n}\equiv 1 \pmod p$. We will extensively use Legendre symbols and all the congruences are modulo $p$ unless otherwise stated.

We adapt the classical Chebyshev polynomials $T_{n}$ and $U_{n}$ ( for more information see \cite{MOS66} section $5.7$, \cite{AbSt72} chapter $22$) by defining
 \begin{align} \label{2.1} t_{n} (x)=t_{n} (x;s)=2s^{n/2} T_{n}
(\frac{x}{2\sqrt{s} }),\end{align}
\begin{align} \label{2.2} u_{n} (x)=u_{n} (x;s)=s^{n/2} U_{n}
(\frac{x}{2\sqrt{s} })\end{align}
for $n\in \N_{0}$ where $s$ is the norm of an algebraic integer in the quadratic field. These are unimodular polynomials with integer coefficients. For technical 
reasons we use this modification of Chebyshev polynomials in order to treat the case $d\equiv 1 \pmod 4$ together with $d\equiv 2,3 \pmod 4$. For the properties 
of this adapted polynomials see Section $6$. Then we specialize the results in the paper  \cite{BiPom} about GL$(2,\Z)$ to quadratic fields. For 
previous works on this subject see e.g. \cite{De79}, \cite{Jar05},\cite{Ki89}.

In Section $2$, we consider matrices $A\in \text{GL}(2,\Z)$ and how the integers of any quadratic field $\mathbb{Q}(\sqrt{d}$ $)$ can be embedded in GL$(2,\Z)$. We also 
prove that $\alpha^{n}\equiv 1 \pmod p$ holds if and only if $A^{n}\equiv I \pmod p$. In the next sections we consider the algebraic integers of 
Norm$(\alpha)\not\eq 0$ and Norm$(\alpha)=\pm1$ respectively. In these sections we apply 
the results of \cite{BiPom} to the quadratic field case. In Section $5$, for a given conductor $f$, we give upper estimates for 
\begin{align*}
 n(f)\coloneqq min\{\nu\in\N: \alpha^\nu\in \mathcal{O}_{f}\}
\end{align*}
and $n(fp^{k})$ for $ k\in\N$ and the odd prime $p.$
\section{The Embedding of Algebraic Integers of $ \mathbb{Q} (\sqrt{d}$ $ )$  }
Let $ A\in $ GL$ (2,\mathbb{C} )$ , that is
\begin{align} \label{3.1} A=\begin{pmatrix} a&b \\ c&d \end{pmatrix},\ \
a,b,c,d\in \mathbb{C} ,\ ad-bc\neq 0.\end{align}
We always write
\begin{align} \label{3.2} x:= \text{tr } A=a+d,\ \ s:= \text{ det }
A=ad-bc.\end{align}

\begin{Pro} \label{T3.1}  For $ n\in \mathbb{N} $ we have
\begin{align} \label{3.3} A^{n} =u_{n-1} (x)A-su_{n-2} (x)I,\end{align}
\begin{align} \label{3.4} A^{n} =\tfrac{1}{2} t_{n} (x)I+u_{n-1}
(x)(A-\tfrac{1}{2} xI).\end{align}
\end{Pro}

This proposition is known in various forms. For instance, (\ref{3.3})  with
$ s=1$  
is Lemma $3.1.3$ in \cite{MaRe03} where $ p_{n} =u_{n-1}$  and $ q_{n}
=u_{n-2}$ . The last
matrix in (\ref{3.4})  has zero trace and it follows that
\begin{align} \label{3.5}  \text{tr } A^{n} =t_{n} (x).\end{align}
With the notation (\ref{3.1})  we can write (\ref{3.4})  as
\begin{align} \label{3.6} A^{n} =\begin{pmatrix} \tfrac{1}{2} t_{n}
(x)+\tfrac{1}{2} (a-d)u_{n-1} (x)&bu_{n-1} (x) \\ 
cu_{n-1} (x)&\tfrac{1}{2} t_{n} (x)-\tfrac{1}{2} (a-d)u_{n-1} (x)
\end{pmatrix}.\end{align}
Now, we consider the algebraic integers $\alpha$ of $\mathbb{Q}(\sqrt{d}$ $)$ using the notation (\ref{1.1}). We define a homomorphism $ \varphi $ 
of the multiplicative group of integers $ \alpha \neq 0$  into GL$ (2,\mathbb{Z} )$. For $ r=2,3$  we set (see e.g. \cite[p. $38$]{Ba03})
\begin{align} \label{3.8} \varphi (\alpha ):=A= \text{ } \begin{pmatrix} a&b
\\ bd&a \end{pmatrix}\end{align}
whereas for $ r=1$  we set
\begin{align} \label{3.9} \varphi (\alpha ):=A=\begin{pmatrix} \tfrac{1}{2}
(a+b)&b \\ qb&\tfrac{1}{2} (a-b) \end{pmatrix} \text{ .}\end{align}
It can be checked that this indeed defines an injective homomorphism.
We have
\begin{align} \label{3.10} s= \text{ det } A= \text{ Norm(} \alpha  \text{)}
=\begin{cases}
a^2 -b^2 d \text{ } & \text{if } r=2,3\\  \text{ } 
\tfrac{1}{4} (a^2 -b^2 d) \text{ } & \text{if } r=1, \text{ } 
\end{cases}\end{align}
\begin{align} \label{3.11} x= \text{ tr } A=\begin{cases}
2a& \text{if } r=2,3\\ 
a& \text{if } r=1.
\end{cases}\end{align}
Since $ A^{n} =\varphi (\alpha ^{n} )$ and $ \varphi $ is injective, it
follows from (\ref{3.6})  that
\begin{align} \label{3.12} \alpha ^{n} =\begin{cases} \text{ } 
\tfrac{1}{2} t_{n} (2a)+u_{n-1} (2a)b\sqrt{d} & \text{if } r=2,3\\ 
\tfrac{1}{2} t_{n} (a)+\tfrac{1}{2} u_{n-1} (a)b\sqrt{d} & \text{if } r=1.
\end{cases}\end{align}
\begin{Pro}
 If $p$ is an odd prime and $\alpha_{k}$, $\alpha_{m}$ are integers of $\mathbb{Q}(\sqrt{d}$ $)$ then $\alpha_{k}\equiv \alpha_{m} \pmod p \text{ if and only if } \varphi(\alpha_{k})\equiv \varphi(\alpha_{m}) \pmod p $.
\end{Pro}
\begin{proof}
 We prove only the more complicated case $r=1$. It can be proved in a similar way for $r=2,3$.

First we assume $\alpha_{k}\equiv \alpha_{m} \pmod p$ and we prove
$\varphi(\alpha_{k})\equiv \varphi(\alpha_{m}) \pmod p $. If $\alpha_{k}\equiv \alpha_{m} \pmod p$, this means
$$
\alpha_{k}=\frac{1}{2}\left(a_{k}+b_{k}\sqrt{d}\right), \ \alpha_{m}=\frac{1}{2}\left(a_{m}+b_{m}\sqrt{d}\right).
$$
Hence, $a_{k}\equiv a_{m}$ and $b_{k}\equiv b_{m} \pmod p$.
Then, we can write $a_{k}+b_{k}\equiv a_{m}+b_{m}$ and  $a_{k}-b_{k}\equiv a_{m}-b_{m}\pmod p$. From congruence properties we can write 
\begin{equation*}
\frac{1}{2} \left(a_{k}+b_{k}\right)\equiv \frac{1}{2}\left(a_{m}+b_{m}\right), \text{   }
\frac{1}{2} \left(a_{k}-b_{k}\right)\equiv \frac{1}{2}\left(a_{m}-b_{m}\right) \pmod p.
\end{equation*}
Then from (\ref{3.9}), $\varphi(\alpha_{k})\equiv \varphi(\alpha_{m}) \pmod p. $

Now we assume $\varphi(\alpha_{k})\equiv \varphi(\alpha_{m}) \pmod p$ and prove $\alpha_{k}\equiv \alpha_{m} \pmod p$. Using the definition in (\ref{3.9}) we can write
\begin{align*}  \varphi (\alpha_{k} ):=\begin{pmatrix} \tfrac{1}{2}
(a_{k}+b_{k})&b_{k} \\ qb_{k}&\tfrac{1}{2} (a_{k}-b_{k}) \end{pmatrix} \text{ }\end{align*}
and similary, $\varphi (\alpha_{m} ) $.
Then, for modulo $p,$ $b_{k}\equiv b_{m}$, $\frac{1}{2} \left(a_{k}+b_{k}\right)\equiv \frac{1}{2}\left(a_{m}+b_{m}\right) $ and 
$\frac{1}{2} \left(a_{k}-b_{k}\right)\equiv \frac{1}{2}\left(a_{m}-b_{m}\right)$. We can obtain $a_{k}\equiv a_{m}$ and this means 
$\alpha_{k}\equiv \alpha_{m}$.
\end{proof}
\begin{Pro}\label{2.xx}
If $p\nmid b$, $p\nmid d$ then $\alpha^{n}\equiv 1\pmod p$ if and only if $A^{n}\equiv I \pmod p$.
\end{Pro}
\begin{proof}
$ \Rightarrow$ First, for modulo $p$, we assume $\alpha^{n}\equiv 1$. For $r=2,3$,
$$\alpha^{n}=\frac{1}{2} t_{n} (x)+u_{n-1} (x)b\sqrt{d}\equiv 1 $$ 
with $p\nmid b$, $p\nmid d$ where $x$ is defined in (\ref{3.11}). 
Since $u_{n-1} (x)\equiv 0  $ by (\ref{3.12}), 
$\frac{1}{2} t_{n} (x)\equiv 1$. Hence, $ A^{n} =\frac{1}{2} t_{n} (x)I+u_{n-1}(x)(A-\tfrac{1}{2} xI)\equiv I$. For $r=1$ namely, 
$ \alpha^{n}=\frac{1}{2} t_{n} (x)+\tfrac{1}{2} u_{n-1} (x)b\sqrt{d} $ proof is similar.

$\Leftarrow$ We assume $ A^{n}\equiv I \pmod p$. Namely, for modulo $p$, 
$$ A^{n} =\frac{1}{2} t_{n} (x)I+u_{n-1}(x)(A-\tfrac{1}{2} xI)\equiv I$$ and 
we aim to prove $\alpha^{n}=\frac{1}{2} t_{n} (x)+u_{n-1} (x)b\sqrt{d}\equiv 1 $ for $r=2,3$. Since by (\ref{3.6}) $bu_{n-1}(x)\equiv 0\pmod{p}$, and as $b\not\equiv 0 $,
thus $ u_{n-1}(x)(A-\tfrac{1}{2} xI)\equiv 0$ and 
tr$(A-\tfrac{1}{2} xI)\equiv 0$ this means
$$
u_{n-1}(x)\begin{pmatrix} \ast & b\\ bd&\ast \end{pmatrix}\equiv 0.
$$
Thus, $u_{n-1}(x)b\equiv 0$. According to our assumption $b\not\equiv 0 \pmod p$ and $u_{n-1}(x)\equiv 0\pmod p$. Therefore from (\ref{3.6}), 
$\frac{1}{2} t_{n} (x)\equiv 1 \pmod p$, for the cases $r=2,3 $ and $r=1$, $\alpha^{n}\equiv 1\pmod p$.
\end{proof}
\section{Algebraic Integers with Norm $\neq 0 $}
In this section, we specialize the results of \cite{BiPom} to the quadratic field case, using the embedding discussed in Section $2$. We allow $d$ to be negative. 
Again we write $d=4q+r$ and $s=\text{Norm}(\alpha)$ where $\alpha$ is an integer of $\mathbb{Q}(\sqrt{d}$ $)$ as in (\ref{1.1}). 

Let $p$ be an odd prime. All the following congruences will be modulo $p$. In the next theorem, $\alpha$ is any algebraic integer with Norm $(\alpha)\neq 0 $ of the form 
(\ref{1.1}). We assume that $p\nmid d,\ p\nmid b$ and 
\begin{align}\label{3.x}
a^2-4s\not\equiv 0 \text{  if  } r=2,3, \text{  } a^{2}-s\not\equiv0 \text{ if } r=1.
\end{align}
Throughout the paper let $x$ be the trace defined by (\ref{3.11}) in the quadratic field case and $s$ be the norm of $\alpha$. Since $t_{n}$ and $u_{n}$ 
are polynomials with integer coefficients the identities in Section $6$ will become valid congruences. We always define $\ell$ as the Legendre symbol 
\begin{align} \label{4.2} \ell:=\left(\frac{x^2-4s}{p}\right)\end{align}
and write $p-\ell$ which is thus $=p\mp 1$ for $\ell=\pm 1$.
\begin{The}
 Let $p$ be an odd prime with $p\nmid d,\ p\nmid b$ and $s=\text{N}(\alpha)\neq 0$. Let $\ell$ be the Legendre symbol defined above. We set 
$\sigma=1$ for $\ell=+1$ and $\sigma=s$ for $\ell=-1$. Then
$$
t_{p-\ell}(x)=2\sigma,\  u_{p-\ell-1}(x)\equiv 0.
$$
We sum up the further results in the following table. 
\begin{center}
\renewcommand\arraystretch{2}

\begin{tabular}{c|c|c}
  & $r=2,3$ &$r=1$\\ \hline
\multirow{2}{*}{$ \left(\frac{s}{p}\right)=+1$ } & $t_{\frac{p-\ell}{2}}(2a)^2\equiv 4\sigma,$& $t_{\frac{p-\ell}{2}}(a)^2\equiv 4\sigma,$\\
&$u_{\frac{p-\ell}{2}-1}(2a)\equiv 0$ & $ u_{\frac{p-\ell}{2}-1}(a)\equiv 0$\\ \hline
\multirow{2}{*}{$\left(\frac{s}{p}\right)=-1$}&$ t_{\frac{p-\ell}{2}}(2a)\equiv 0,$&$ t_{\frac{p-\ell}{2}}(a)\equiv 0, $\\
& $ (a^2-4s)u_{\frac{p-\ell}{2}-1}(2a)^2\equiv 4\sigma $& $(a^2-s)u_{\frac{p-\ell}{2}-1}(a)^2\equiv \sigma   $\\
\end{tabular}
\end{center}
\end{The}

This is \cite[Th.$4.1$]{BiPom} specialized to our present situation. 

The proof in \cite{BiPom} used Chebyshev polynomials. In the present context of quadratic fields, many of the above formulas can be proved by other methods, see 
for instance \cite{Ba03}, \cite[Th.$1.7$ ]{Le30}.
\section{Algebraic Integers with Norm $\pm1 $}
First we consider the case that $s=\text{Norm}(\alpha)=+1$. All congruences will be modulo the odd prime $p$. Also $\ell$ is the Legendre symbol defined in 
(\ref{4.2}) and $x$ is defined in (\ref{3.11}).

The following results are obtained by specializing the results in Section $5$ and $6$ of \cite{BiPom}. The Legendre polynomials $t_{n}$ and $u_{n-1}$ depend only 
on $x$ and $s$ as defined in (\ref{3.10}) and (\ref{3.11}); the specific form (\ref{1.1}) of $\alpha$ is not important.
\begin{Pro} \label{T5.1 } Let $ k\in \mathbb{N} $  divide $ p-\ell$  and we assume that $ \ell=(\tfrac{x^2 -4s}{p})\neq 0$. If $ x\equiv t_{k} (y)$  for
some $ y\in \mathbb{Z} $  then, with $ n=\frac{p-\ell}{k}$ ,
\begin{align} \label{5.1} t_{n} (x)\equiv 2,\ \ u_{n-1} (x)\equiv 0,\ \
\alpha^{n} \equiv 1.\end{align}
\end{Pro}

See \cite[Th. $5.1$]{BiPom}.

For the special case that $k=2^j$ we can say much more. We construct $ x_{0} ,\dots,x_{m}$ 
recursively by the following rule. Let $ x_{0} =x$. If $
(\frac{x+2}{p})=-1$ 
then we set $ m=0$  and stop.
\ 
Now let $ (\frac{x+2}{p})=+1$  and suppose that $ x_{0} ,\dots,x_{k}$  have
already been
constructed such that $ 2^{k} |p-\ell$  and
\begin{align} \label{5.4} x_{\nu -1}\equiv t_{2} (x_{\nu } ),\ ((x_{\nu}^2 -4)/p)=\ell
\text{\ \ \ for } 1\leq \nu \leq k.\end{align}
If $ 2^{k+1} \ \nmid\ p-\ell$  or if $ (\frac{x_{k}+2}{p})$ $ =-1$  then we set $
m=k$  and stop. 
Otherwise we have $ 2^{k+1} |p-\ell$  and $ (\frac{x_{k}+2}{p})=+1$. Then there  
exists $ x_{k+1}$  such that $ x_{k} +2\equiv x^2 _{k+1}$  and thus $ x_{k}
=t_{2} (x_{k+1} ).$ 
It follows from (\ref{5.4})  that
\begin{align*}
((x_{k} -2)/p)=((x_{k} +2)/p)((x_{k} -2)/p)=((x^2 _{k} -4)/p)=\ell\end{align*}
and therefore $ ((x^2 _{k+1} -4)/p)=((x_{k}-2)/p)=\ell$ . This completes our
construction. Note that $ 2^{m} \ |\ p-\ell.$ 
\begin{The}
 Let N$(\alpha)=1$ and $ \ell=(\tfrac{x^2 -4}{p})\neq 0$
and let $ x_{0} ,\dots,x_{m} $ be as
constructed above. Then
\begin{align} \label{5.5} t_{(p-\ell)/2^k}(x)\equiv 2 \text{\ \ \ for }
k=0,\dots,m,\end{align}
\begin{align} \label{5.6} t_{(p-\ell)/2^{m+1}}(x) \equiv -2 \text{\ \ or\ \ }
2^{m+1} \ \nmid\  p-\ell.\end{align}
\end{The}
See \cite[Th. $4.3$]{BiPom}.
\begin{Cor}\label{T5.5}  Let $s=\text{N}(\alpha)=1$  and $ \ell=(\tfrac{x^2 -4}{p})\neq 0$
and let $ x_{0} ,\dots,x_{m}$  be
constructed as above. Writing $ n=(p-\ell)/2^{m}$  we have
\begin{align} \label{5.8} u_{n-1} (x)\equiv 0,\ \alpha^{n} \equiv 1.\end{align}
If $ 2^{m+1} \ |\ p-\ell$  then furthermore
\begin{align} \label{5.9} u_{\frac{n}{2}-1} (x)\equiv 0,\ \alpha^{n/2} \equiv
-1.\end{align}
These bounds are best possible: If $2^{m+2}\mid p-\ell$ then $u_{\frac{n}{2}-1} (x)\not\equiv 0$.

\end{Cor}

\begin{proof}
 Since $s=1$ and for modulo p, $x^2 -4\not\equiv 0$ it follows from (\ref{2.4}) and (\ref{5.5}) that $u_{n-1}\equiv 0$  and therefore $A^{n}\equiv I$ by 
(\ref{3.4}). By Proposition \ref{2.xx} we have $\alpha^{n}\equiv 1.$ This proves (\ref{5.8}). Furthermore if $ 2^{m+1} \ |\ p-\ell$ then (\ref{5.9}) follows from (\ref{5.6})
 by the same argument. Now let $2^{m+2}\mid p-\ell$. Then it follows from (\ref{5.3}) that $t_{n/2}(x)\equiv -2$ so that $t_{n/4}(x)\equiv 0$ by recursion formulas of 
$t_n(x)$ which is similar to $u_{n}(x)$ in Section $6$. Hence $u_{\frac{n}{4}-1}(x)\not\equiv 0$.
\end{proof}
Now we consider the more complicated case that N$(\alpha)=-1$ so that in general $t_{n}(x)=t_{n}(x;-1)$. As before we set $ \ell:=\left(\frac{x^2-4s}{p}\right)$ 
and assume that (\ref{3.x}) with $s=-1$ holds. All congruences will be modulo the odd prime $p$.
Since $ (-1/p)=(-1)^{(p-1)/2}$ we obtain from Theorem $3.1$ (where now $\sigma=\ell$)
that,
with $ n=\frac{p-\ell}{2}$ ,
\begin{align} \label{6.1} t_{2n} (x)\equiv 2\ell,\ t_{n} (x)^2 \equiv 4\ell,\
u_{n-1} (x)\equiv 0 \text{\ \ \ for } p=4q+1,\end{align}
\begin{align} \label{6.2} t_{2n} (x)\equiv 2\ell,\ t_{n} (x)\equiv 0,\ u_{n-1}
(x)\not\equiv 0 \text{\ \ \ for } p=4q+3,\end{align}
and it follows from (\ref{6.3x})  that
\begin{align} \label{6.3} t_{2(p-\ell)} (x)\equiv 2.\end{align}
Hence $ t_{n} (x)\equiv \pm 2$  holds if and only if $ p=4q+1$  and $ \ell=+1$
, which we now assume.
\ 
It follows from (\ref{2.12}) and $ t_{2} (x;-1)=x^2 +2$  that
\begin{align} \label{6.4} t_{2n} (x;-1)=t_{n} (x^2 +2;1) \text{\ \ \ for }
n\in \mathbb{N} .\end{align}
Since $ (\frac{-1}{p})=+1$  there exists $ j\in \mathbb{Z} $  with $ j^2
\equiv -1$ . We now assume that $ x\not\equiv 0$ 
and $ x\not\equiv \pm 2j.$  Then
\begin{align} \label{6.5} (x^2 +2)^2 -4=x^2 (x^2 +4)\not\equiv 0.\end{align}
As in section $4$, we construct numbers $ y_{0} ,\dots,y_{m}$  with the only
difference that $ y_{0} =x^2 +2$  instead of $ x_{0} =x$ . It follows from
(\ref{6.5})  that
also $ ((y^2 _{0} -4)/p)=\ell.$  We have $ y_{0} +2=x^2 +4$  and thus $ ((y_{0}
+2)/p)=\ell=+1.$ 
Hence the first step of our construction can 
be done so that $ m\geq 1$. The construction stops if $ ((y_{m} +2)/p)=-1$ 
or $ 2^{m+1} \ \nmid\ p-1.$ 
\begin{The}\label{T6.1} Let N$(\alpha)=-1,p=4q+1,\ a^2 +4\not\equiv 0$, $\ell=+1$ and 
let $ y_{0} ,\dots,y_{m}$  be
constructed as above. Then $ m\geq 1$  and
\begin{align} \label{6.6} t_{(p-1)/2^{k}}(x)\equiv 2\ \ \  \text{for }
k=0,\dots,m-1,\end{align}
\begin{equation}
t_{(p-1)/2^{m}}(x)\equiv \left\{
\begin{array}{ll}
 \displaystyle -2,\ \text{or}\\
\displaystyle 0 \ \text{and}\ 2^{m+1}\nmid p-\ell.
\end{array}
\right.
\end{equation}
\end{The}
See \cite[Th.6.1]{BiPom}. The next result is not a surprise because N$(\alpha^2)=1$. Its proof is similar to the proof of Corollary $4.3$.
\begin{Cor}
 Under the assumptions of Theorem $4.4$, we now write $n=(p-\ell)/2^{m-1}$. Then (\ref{5.8}) holds, and if $2^{m+1}\mid p-\ell$ then (\ref{5.9}) also holds. These 
bounds are best possible: If  $2^{m+1}\mid p-\ell$ then $u_{\frac{n}{4}-1}(x)\not\equiv 0.$
\end{Cor}

\begin{The}\label{T6.2}  Let N$ (\alpha)=-1$ and let $ k$  be odd with $ k\ |\
p-\ell.$  If $ x^2 +4\not\equiv 0$ 
and $ x\equiv t_{k} (y;-1)$  for some $ y\in \mathbb{Z} $  then, with $
n=(p-1)/k,$ 
\begin{align} \label{6.8} t_{2n} (x)\equiv 2,\ \ t_{n} (x)\equiv 2\ell,\ \
\alpha^{n} \equiv \ell .\end{align}
 
\end{The}
\begin{proof}
 Proof is given more generally in \cite{BiPom}.
\end{proof}
\section{Estimates for Conductors}
We continue to study the quadratic field $ \mathbb{Q} (\sqrt{d}$ $ )$ where $d\equiv r=1,2,3 \pmod 4$. The order of the conductor $f\in\N$ is 
\begin{align}\label{5.1x}
\mathcal{O}_{f}= \left\{\begin{array}{ll} {\{a'+b'f\sqrt{d}:\ a',b'\in\Z\}} & \text{if } r=2,3, \\
                                 \{\frac{1}{2}(a'+(f-1)b')+\frac{1}{2}b'f\sqrt {d}:\ a',b'\in\Z, \ 2\mid a'+b'\} &  \text{if }  r=1.
                                 \end{array}\right. 
\end{align}
We fix an integer $\alpha$ of $ \mathbb{Q} (\sqrt{d}$ $ )$ with $s=\text{N}(\alpha)\neq 0$ and $x$ is given by (\ref{3.11}); again we use the notation (\ref{1.1}). 
The most interesting case is that $\alpha$ is the fundamental unit of $ \mathbb{Q} (\sqrt{d})$. Following Halter-Koch we define 
%\begin{subequations}
\begin{align}\label{5.2}
 n(f)=n(f,\alpha)\coloneqq min \{\nu\in\N :\ \alpha^{\nu}\in\mathcal{O}_{f}\}.
\end{align}
\begin{Lem}
Let $b\not\eq 0$ be given by (\ref{1.1}) and $s,x$ by (\ref{3.10}). We write 

\begin{align}\label{5.2a}
c\coloneqq gcd(b,f), b_0\coloneqq b/c, f_0=f/c 
\end{align}
%\end{subequations}
Then we have 
\begin{align}\label{5.3}
  n(f)=n(f_0)=min \{\nu\in\N :\ u_{\nu-1}(x;s)\equiv 0 \pmod {f_{0}}\}.
\end{align}
\end{Lem}
\begin{proof}
By (\ref{3.10}) and (\ref{3.11}) we have 
\begin{align*}
 \alpha^\nu\in\mathcal{O}_{f} \Leftrightarrow bu_{\nu-1}(x)\equiv 0 \pmod f.
\end{align*}
Since gcd$(b_0,f_0)=1$ by (\ref{5.2a}), it follows that 
\begin{align*}
\alpha^\nu\in\mathcal{O}_{f} \Leftrightarrow b_0u_{\nu-1}(x)\equiv 0 \pmod {f_{0}}\Leftrightarrow u_{\nu-1}(x)\equiv 0. \pmod {f_{0}}
\end{align*}
Note that $b$ has not been replaced by $b_0$. Therefore we still have $u_{\nu-1}(x)=u_{\nu-1}(x;s)$ with unchanged $x$ and $s$. 
\end{proof}
% Now we assume that gcd$(b,f)=1$ where $b$ is defined in (\ref{1.1}). Then we obtain from (\ref{5.1x}),(\ref{5.2}) and (\ref{3.12}) that
Let $g\in\N$ and gcd$(b,g)=\text{gcd}(f,g)=1$. Then it follows from (\ref{5.3}) and (\ref{6.5x}) that $u_{n(f)n(g)-1}(x;s)\equiv 0 $ modulo $f$ and also modulo $g$. 
Hence we have 
\begin{align}\label{5.4x}
 n(fg)\leq n(f)n(g)\text{\    if\    gcd}(f,g)=1. 
\end{align}
For an odd prime $p$ we define 
\begin{align}\label{5.5xx}
q(p)=q(p;\alpha)\coloneqq  min \{\nu\in\N :\ u_{\nu-1}(x;s)\equiv 0\pmod p\}.
\end{align}
The results of Sections $3$ and $4$ provide upper estimates of $q(p)$. These results refer only on $x$ and $s$ and not explicitly on $a,b$ and $d$ in (\ref{1.1}).

First let $ \ell=\left(\frac{x^2-4s}{p}\right)\neq 0.$ If $s=1$ it follows from 
Corollary $4.2$ that
\begin{align*}
 q(p)\leq \frac{p-\ell}{2^m},\text{ and }  q(p)\leq \frac{p-\ell}{2^{m+1}}\text{  if  } 2^{m+1}\mid p-\ell.
\end{align*}
If $s=-1$, $p\equiv 1\pmod 4$ and $\ell=+1$ then it follows from Corollary $4.4$ that 
\begin{align*}
 q(p)\leq \frac{p-\ell}{2^{m-1}},\text{ and }  q(p)\leq \frac{p-\ell}{2^{m}}\text{  if  } 2^{m}\mid p-\ell.
\end{align*}
Now let $x^2-4s\equiv 0 \pmod p.$ Then it follows from (\ref{2.4}) that $2^{\nu-1}u_{\nu-1}(x;s)\equiv \nu x^{\nu-1}\pmod p.$ We conclude that $q(p)=p$ if $p\nmid s$
and $q(p)=2$ if $p\mid s$.
\begin{The}
 If gcd$(f,b)=1$ and $p\nmid f$ then 
\begin{align}\label{5.6x}
 n(p^kf)\leq q(p)p^{k-1}n(f)\text{ for } k\geq 1.
\end{align}
\end{The}
\begin{proof}
 We use induction on $k$. By (\ref{5.3}) and (\ref{6.5x}) we have $u_{q(p)n(f)-1}(x;s)\equiv 0$ modulo $f$, but also modulo $p$ by (\ref{5.5xx}) and (\ref{6.5x}). 
Since gcd$(f,p)=1$ it follows that the congruence is true also modulo $pf$. Hence (\ref{5.6x}) holds for $k=1$ in view of (\ref{5.3}).

Now let (\ref{5.6x}) hold for $k$. We write $\nu=q(p)p^{k-1}n(f)$ and have, by (\ref{5.6x}),
\begin{align}\label{5.7x}
u_{\nu-1}(x;s)\equiv 0 \pmod {p^kf}.       
\end{align}
We apply (\ref{2.4}) with $n=p$ and with $s^\nu$ instead of $s$. The binomial coefficients in the sum are divisible by the prime $p$. Since $2^{p-1}\equiv 1\pmod p$
we thus obtain for $z\in \Z$ that 
\begin{align*}
 u_{p-1}(z;s^\nu)\equiv (z^2-4s^\nu)^{(p-1)/2}\pmod p.
\end{align*}
For $z=t_{\nu}(x;s)$ we obtain by (\ref{6.2x}), that 
\begin{align}\label{5.8x}
 u_{p-1}(t_\nu(x;s);s^\nu)\equiv\left[(x^2-4s)u_{\nu-1}(x;s)\right]^{\frac{p-1}{2}}\equiv 0 \pmod p
\end{align}
because of (\ref{5.7x}) where $k\geq 1.$ Now we apply (\ref{6.4x}) with $m=p$ and $n=\nu.$ By (\ref{5.7x}) and (\ref{5.8x}) we obtain 
\begin{align*}
 u_{q(p)p^k-1}(x;s)=u_{p\nu-1}(x;s)\equiv 0 \pmod {p^{k+1}f}.
\end{align*}
Hence it follows from (\ref{5.3}) that $n(p^{k+1}f)\leq q(p)p^k.$
\end{proof}
% {\it Remark}.
%  The condition gcd$(f,b)=1$ of Theorem $5.1$ is rather restrictive. In some special cases it can be relaxed as we now show. Let $c:=$gcd$(f,b)$ and 
% $b_0=b\slash c$, $f_0=f\slash c$. Let $\alpha_0$ be defined by (\ref{1.1}) with $b$ replaced by $b_0$ and write $d=4q+r$; if $r=1$ it is necessary that $b-b_0$ is even. 
% We assume now that $c=kf_0\pm 1$ with $k\in \Z$ if $r=2,3$ and $k\in 2\Z$ if $r=1$. Then we have 
% \begin{equation}\label{5.9xxx}
% b=(kf_0\pm 1)b_0, \  f=(kf_0\pm 1)f_0, \ \text{gcd}(b_0,f_0)=1. 
% \end{equation}
% This condition is rather artificial but it is the best that can be obtained by the our method of Chebyshev polynomials.
% 
% We show now that $n(f,\alpha)=n(f_0,\alpha_0)$. By (\ref{3.10}) we have, if $r=2,3$,
% \begin{equation*}
%  s=\text{N}(\alpha)=a^2-(kf_0\pm 1)^2{b_{0}}^2d\equiv a^2-b_0^2d=\text{N}(\alpha_0)=s_0
% \end{equation*}
% modulo $f_0$. If $r=1$ we also obtain $s\equiv s_0 \pmod {f_0}$ because now $k$ is even. Let $\alpha ^\nu\in\mathcal{O}_{f}$. 
% Then by (\ref{5.1x}), (\ref{5.2}) and (\ref{3.12}), we have $u_{\nu-1}(x;s)b\equiv 0\pmod f$ and thus $u_{\nu-1}(x;s)b_0\equiv 0 \pmod {f_0}$. 
% Since $s\equiv s_0\pmod {f_0}$
% and since $u_{\nu-1}\in \Z\left[x,s\right]$ it follows that $u_{\nu-1}(x;s)\equiv u_{\nu-1}(x;s_0)\pmod {f_0}$. Hence $\alpha_0 ^\nu\in\mathcal{O}_{f_{0}}$. 
% The converse is trivial.

\begin{The}
 Let $f\in\N$ be odd and let $f_0$ be defined by (\ref{5.2a}). We write 
\begin{align}\label{5.9x}
 f_0=\prod_{\nu=1}^{\mu}{p_\nu}^{k_\nu}\text{           }\       (k_\nu\in\N)
\end{align}
with different primes $p_\nu$. Then 
\begin{align}\label{5.10x}
 n(f)\leq \prod_{\nu=1}^{\mu}\left(q(p_\nu){p_\nu}^{k_\nu-1}\right).
\end{align}
\end{The}
\begin{proof}
 Let $g_0=1$ and, for $1\leq\lambda\leq\mu,$ let $g_\lambda$ be the subproduct of (\ref{5.9x}) for $\nu=1,\dots,\lambda.$ Then $g_\lambda=p^{k_\lambda}g_{\lambda-1}$
and $p_\lambda\nmid g_{\lambda-1}$. Hence we obtain from Theorem $5.2$ applied to $f_0$ that 
\begin{align*}
 n(f_{\lambda})\leq q(p_\lambda)p^{{k_\lambda}-1}n(f_{\lambda-1}),
\end{align*}
and (\ref{5.10x}) with $f$ replaced by $f_0$ follows by induction. Finally we use that $n(f)=n(f_0)$ by Lemma $5.1.$
\end{proof}
\section{Some Formulas for Chebyshev Polynomials}
We gather some formulas that we need to prove our results, concentrating on the polynomials $u_n$ defined in (\ref{2.2}). See \cite[Sect.$5.7$,]{MOS66} and \cite{BiPom}.
 For odd $n$ and $x,s\in\C$, we have 
\begin{align}\label{2.4} 
u_{n-1}(x;s)=\frac{1}{2^{n-1}}\sum_{k=0}^{(n-3)/2}
\binom{n}{2k+1}x^{n-2k-1}(x^2 -4s)^{k}+\frac{1}{2^{n-1}}(x^2 -4s)^{\frac{n-1}{2}}.
\end{align}
The recursion formula $u_{n+1}(x)=xu_n(x)-su_{n-1}(x)$ shows that, for $x,s\in\C$,
\begin{align*}
u_{0}(x)=1,\  u_{1}(x)=x,\ u_{2}(x)=x^2-s,\  u_{3}(x)=x^3-2sx,\ \\
u_{4}(x)=x^4-3sx^2+s^2,\ u_{5}(x)=x^5-4sx^3+2s^2x.
\end{align*}
Furthermore that $t_n(x;s)$, $u_n(x;s)\in\Z\left[x,s\right]$. For $n\in\N$ we have 
\begin{align}\label{6.2x}
(x^2-4s) u_{n-1}(x;s)^2=t_{n}(x;s)^2-4s^n,              
\end{align}
\begin{align}\label{6.3x}
 t_n(x;s)^2=t_{2n}(x;s)+2s^n.
\end{align}
We need a relation for products which involves different parameters.
\begin{align}\label{6.4x}
 u_{mn-1}(x;s)= u_{m-1}(t_n(x;s);s^n)u_{n-1}(x;s)\   (m,n\in\N).
\end{align}
It follows that, for $\mu\in\N$ and $x,s\in\Z$,
\begin{align}\label{6.5x}
 u_{n-1}(x;s)\equiv 0 \pmod \mu\Rightarrow u_{mn-1}(x;s)\equiv 0 \pmod \mu.
\end{align}
To prove (\ref{6.4x}) it is sufficient to consider $\frac{x}{2\sqrt s}=\cos\theta$ with real $\theta.$ Then it follows from (\ref{2.1}), (\ref{2.2}) and the properties 
\cite[p.$257$,]{MOS66} of the $T_n$ and $U_n$ that 
\begin{align}\label{6.6x}
 t_{n}(x;s)=2s^{\frac{n}{2}}\cos(n\theta),\text{  } u_{m-1}(x;s)=s^{\frac{m-1}{2}}\frac{\sin(m\theta)}{\sin \theta}.
\end{align}
By (\ref{2.2}) and (\ref{2.1}) we therefore have 
\begin{align*}
 u_{m-1}(t_n(x;s);s^n)&=s^{n\frac{m-1}{2}}U_{m-1}(\frac{1}{2s^{n/2}t_n(x;s)})\\
&=s^{n\frac{m-1}{2}}U_{m-1}(\cos(n\theta))=s^{\frac{mn-n}{2}}\frac{\sin(mn\theta)}{\sin n\theta}.
\end{align*}
Now we multiply by $u_{n-1}(x;s).$ Using (\ref{6.6x}) we obtain
\begin{align*}
 u_{m-1}(t_n(x;s);s^n)u_{n-1}(x;s)=s^{\frac{mn-1}{2}}\frac{\sin(mn\theta)}{\sin n\theta}=u_{mn-1}(x;s)
\end{align*}
using again (\ref{6.6x}).

 In Section $4$ we use the following relation between the polynomials $ t_{n} (x;s)$  with
 different parameters $s$. If $ s\neq 0$  and $ m,n\in \mathbb{N} $  then
  \begin{align} \label{2.12} 
t_{mn} (x;s)=t_{n} (t_{m} (x;s);s^{m}
 );\end{align}
indeed it follows from (\ref{2.1})  and the composition property $ T_{mn}
=T_{n} \circ T_{m}$ 
that
\begin{align*}
t_{mn} (x;s)&=2(s^{m} )^{n/2} T_{n} (T_{m} (\frac{x}{2\sqrt{s} }))\\ 
&=t_{n} (\frac{1}{2\sqrt{s} ^{m} }T_{m} (\frac{x}{2\sqrt{s} });s^{m}
)\end{align*}
which implies (\ref{2.12})  by (\ref{2.1}).

\begin{center}
\bf{Acknowledgements}
\end{center}
The authors would like to thank Prof. Dr. F. Halter-Koch for suggesting the interesting problem of Section $5$, Prof. Dr. Christian Pommerenke for the 
identity (\ref{6.4x}) and also Prof. Dr. Attila Peth\H{o} for his valuable comments and remarks.

\end{document}